\documentclass[12pt, reqno]{article} %\documentclass[12pt, reqno]{amsart}
\usepackage{amsfonts}
\usepackage{amsthm}
\usepackage{amscd}
\usepackage{amssymb}
\usepackage{graphics}
\usepackage{amsmath}
\usepackage[english]{babel}
\usepackage{hyperref}
\usepackage{bbm}
\usepackage{yfonts}
\usepackage{mathrsfs}
\usepackage{ upgreek }
\usepackage{color}
\usepackage{epsfig}
\usepackage{graphicx}
\usepackage{subfig}

\usepackage{amsthm}
\usepackage{graphicx}
\usepackage{verbatim}
\usepackage{float}
\usepackage{fullpage}
\usepackage{amsfonts}
\usepackage{amscd}

\allowdisplaybreaks[4]

\DeclareFontFamily{OML}{rsfs}{\skewchar\font'177}
\DeclareFontShape{OML}{rsfs}{m}{n}{ <5> <6> rsfs5 <7> <8> <9>
rsfs7 <10> <10.95> <12> <14.4> <17.28> <20.74> <24.88> rsfs10 }{}
\DeclareMathAlphabet{\mathfs}{OML}{rsfs}{m}{n}
\newtheorem{prop}{Proposition}[section]

\newtheorem{thm}[prop]{Theorem}
\newtheorem{lem}[prop]{Lemma}

\newtheorem{cor}[prop]{Corollary}
\newtheorem{conj}[prop]{Conjecture}

\newtheorem{defn}[prop]{Definition}

\newtheorem{prob}[prop]{Problem}

\newtheorem{clm}[prop]{Claim}

\newcommand{\BE}{{\mathbb{E}}}

\newcommand{\BN}{{\mathbb{N}}}

\newcommand{\BP}{{\mathbb{P}}}

\newcommand{\BR}{{\mathbb{R}}}

\newcommand{\BT}{{\mathbb{T}}}

\newcommand{\BZ}{{\mathbb{Z}}}

\newcommand{\CA}{{\mathcal{A}}}

\newcommand{\CC}{{\mathcal{C}}}
\renewcommand{\CD}{{\mathcal{D}}}

\newcommand{\var}{{\textbf{Var}}}
\newcommand{\cov}{{\textbf{Cov}}}
\newcommand{\ind}{{\mathbbm{1}}}

\renewcommand{\prob}{{\bf P}}
\newcommand{\bae}{\begin{equation}\begin{aligned}}
\newcommand{\eae}{\end{aligned}\end{equation}}
\newcommand{\ev}{\mathbf{E}}
\newcommand{\pr}{\mathbb{P}}
\newcommand{\Z}{\mathbb{Z}}

\newcommand{\om}{{\omega}}
\newcommand{\Om}{{\Omega}}

\newcommand{\ep}{{\epsilon}}

\begin{document}

\numberwithin{equation}{section} \numberwithin{figure}{section}

\title{Concentration estimates for the isoperimetric constant of the super critical percolation cluster }

\author{Eviatar B. Procaccia\footnote{Weizmann Institute of Science}, Ron Rosenthal\footnote{Hebrew University of Jerusalem}}

\maketitle

\begin{abstract}
We consider the Cheeger constant $\phi(n)$ of the giant component of
supercritical bond percolation on $\BZ^d/n\BZ^d$. We show that the
variance of $\phi(n)$ is bounded by $\frac{\xi}{n^d}$, where $\xi$
is a positive constant that depends only on the dimension $d$ and the
percolation parameter.
\end{abstract}
%\tableofcontents
%*****************************************************************************************************************************************************
%***************************************************** Section I - Introduction **********************************************************************
%*****************************************************************************************************************************************************

\section{Introduction}

Let $\BT^d(n)$ be the $d$ dimensional torus  with side length $n$,
i.e,  $\BZ^d/n\BZ^d$, and denote by $\BE_d(n)$
 the set of edges of the graph $\BT^d(n)$. Let $p_c(\BZ^d)$ denote the critical value for  bond percolation on $\BZ^d$, and fix some $p_c(\BZ^d)<p\leq 1$.
We apply a $p$-bond Bernoulli percolation process on the torus $\BT^d(n)$ and denote by $C_d(n)$ the largest open component
of the percolated graph (In case of two or more identically sized largest components, choose one by some arbitrary but fixed method). Let $\Om=\Om_n=\{0,1\}^{\BE_d(n)}$ be the space of configurations for the
percolation process and $\prob=\prob_{p}$ is the probability measure associated with the percolation process. For a subset $A\subset C_d(n)(\om)$ we denote by $\partial_{C_d(n)} A$ the boundary of the
set $A$ in $C_d(n)$, i.e, the set of edges $(x,y)\in \BE_d(n)$ such that $\om((x,y))=1$ and with
either $x\in A$ and $y\notin A$ or $x\notin A$ and $y\in A$.
Throughout this paper $c,C$ and $c_i$ denote positive constants which may depend on the
dimension $d$ and the percolation parameter $p$ but not on $n$.
The value of the constants may change from one line to the next.

Next we define the Cheeger constant

\begin{defn}
For a set $\emptyset\neq A\subset C_d(n)$ we denote,
\[
\psi_A=\frac{|\partial_{C_d(n)} A|}{|A|} .\] where $|\cdot|$ denotes
the cardinality of a set. The Cheeger constant of $C_d(n)$ is
defined by:
\[
\phi=\phi(n):=\min_{\begin{array}{ll}&_{\emptyset\neq A\subset C_d(n)}\\&_{|A|\le|C_d(n)|/2}\end{array}}\psi_A.
\]
\end{defn}

In \cite{benjamini2003mixing} Benjamini and Mossel studied the robustness of the mixing time and Cheeger constant of $\Z^d$ under a percolation perturbation. They showed that for $p_c(\Z^d)<p<1$ large enough $n\phi(n)$ is bounded between two constants with high probability. In \cite{mathieu2004isoperimetry}, Mathieu and Remy improved the result and proved the following on the Cheeger constant
\begin{thm}
There exist constants $c_2,c_3,c>0$ such that for every $n\in\BN$
\[
\prob\left(\frac{c_2}{n}\leq \phi(n)\leq \frac{c_3}{n}\right)\geq 1-e^{-c\log^\frac{3}{2}n}.
\]
\end{thm}

Recently, Marek Biskup and G\'{a}bor Pete brought to our attention that better bounds on the Cheeger constant exist. In \cite{pete2007note} and \cite{berger2008anomalous} it is shown that 
\begin{thm}[\cite{pete2007note}]
For $d\geq 2$ and $p>p_c(\BZ^d)$, there are constants $\alpha(d,p)>0$ and $\beta(d,p)>0$  such that 
\[
\prob\left(\exists S \text{ connected } : 0\in S\subset \mathcal{C}_\infty~,~ M\leq |S|<\infty~,~\frac{|\partial_{\mathcal{C}}S|}{|S|^{(d-1)/d}}\leq \alpha\right)\leq \exp\left(-\beta M^{(d-1)/d}\right)
\]
\end{thm}

The improved bounds don't improve our result, thus we kept the original \cite{mathieu2004isoperimetry} bounds in our proofs. 

\begin{conj}
The limit $\lim_{n\rightarrow\infty}{n\phi(n)}$ exists.
\end{conj}

Even though the last conjecture is still open, and the
expectation of the Cheeger constant is quite evasive, we managed to
give a good bound on the variance of the Cheeger constant. This is
given in the main Theorem of this paper:

\begin{thm}\label{thm:main}
There exists a constant $\xi=\xi(p,d)>0$ such that
\[
\var(\phi)\le\frac{\xi}{n^d}.
\]
\end{thm}
A major ingredient of the proof is Talagrand's inequality for concentration of measure on product spaces. This inequality is used by Benjamini, Kalai and Schramm in \cite{benjamini2003first} to prove concentration of first passage percolation distance. A related study that uses another inequality by Talagrand is \cite{alon2002concentration}, where Alon, Krivelevich and Vu prove a concentration result for eigenvalues of random symmetric matrices.

\section{The Cheeger constant}
Before turning to the proof of Theorem \ref{thm:main}, we give the following definitions:

\begin{defn}\label{def:om^e}
For a function $f:\Om\rightarrow\BR$ and an edge $e\in \BE_d(n)$ we define $\nabla_e f:\Om\rightarrow\BR$ by
\[
\nabla_e f(\om)=f(\om)-f(\om^e)
\]
where
\[
\omega^{e}(e')=\begin{cases}
\omega(e') & \,\, e'\neq e\\
1-\omega(e') & \,\, e'=e\end{cases}.
\]
In addition, for a configuration $\om\in\Om$ and an edge $e\in
\BE_d(n)$, let $\hat{\om}^{e}=\min\{\om, \om^e\}$ and
$\check{\om}^e=\max\{\om,\om^e\}$.
\end{defn}

\begin{defn}
For $n\in\BN$ we define the following events:
\begin{equation}\begin{aligned}
H^1_n(c_1)&=\left\{\om\in\Om ~:~|C_d(n)(\om)|>c_1 n^d\right\}\\
H^2_n(c_2,c_3)&=\left\{\om\in\Om ~:~\frac{c_2}{n}<\phi(n)(\om)<\frac{c_3}{n}\right\}\\
H^3_n&=\left\{\om\in\Om ~:~\forall e\in E_d(n)\quad
\left|C_d(n)(\om)\triangle
C_d(n)(\om^e)\right|\leq\sqrt{n}\right\}\\
H^4_n(c_4)&=\{\om\in\Om ~:~\exists A:|A|>c_4 n^d,\psi_A(\om)=\phi(n)(\om)\}\\
H^5_n(c_5)&=\{\om\in\Om ~:~\exists A:|A|>c_5
n^d,\psi_A(\om^e)=\phi(n)(\om^e)\}\\
\end{aligned},\end{equation}
and
\begin{equation}\begin{aligned}\label{eq:hndef}
H_n &=H_n(c_1,c_2,c_3,c_4,c_5)=H_n^1(c_1)\cap H_n^2(c_2,c_3)\cap
H_n^3\cap H_n^4(c_4)\cap H_n^5(c_5).
\end{aligned}\end{equation}
\end{defn}

We start with the following deterministic claim:

\begin{clm}\label{clm:nabphi}
Given $c_1,c_2,c_3,c_4,c_5>0$, there exists a constant
$C=C(c_1,c_2,c_3,c_4,c_5,d,p)>0$ such that if $\om\in
H_n(c_1,c_2,c_3,c_4,c_5)$ then for every $e\in\BE_d(n)$
\[
|\nabla_e\phi(\om)|\le\frac{C}{n^d} .\]
\end{clm}

In order to prove Claim \ref{clm:nabphi} we will need the following two lemmas:

\begin{lem}\label{lem:isobigset}
Fix a configuration $\om\in\Om$ and an edge $e\in \BE_d(n)$. Let
$A\subset C_d(n)(\hat{\om}^e)$ be a subset such that $|A|=\alpha
n^d$. Then
\[
|\nabla_e\psi_A|\leq \frac{1}{\alpha n^d}.
\]
\end{lem}

\begin{proof}
Since $A$ is a subset of $C_d(n)(\hat{\om}^e)$ it follows that the
size of $A$ doesn't change between the configurations $\hat{\om}^e$
and $\check{\om}^e$ and the size of $\partial_{C_d(n)} A$ is changed
by at most $1$. It therefore follows that
\begin{equation}
\begin{aligned}
|\nabla_e \psi(A)|&=|\psi_A(\omega)-\psi_A(\om^e)|=|\psi_A (\hat{\om}^e)-\psi_A (\check{\om}^e)|\\
&\leq \left|\frac{|\partial A|}{|A|}-\frac{|\partial A|+
1}{|A|}\right|=\frac{1}{|A|}.
\end{aligned}
\end{equation}
\end{proof}

\begin{lem}\label{lem:twocom}
Let $G$ be a finite graph, and let $A,B\subset G$ be disjoint such
that there exists a unique edge $e=(x,y)$, such that $x\in A$ and $y\in B$, then
\[
\psi_{A\cup B}\ge\min\{\psi_A,\psi_B\}-\frac{2}{|A|+|B|} .\]
\end{lem}

\begin{proof}
From the assumptions on $A$ and $B$ it follows that
\begin{equation}
\begin{aligned}
\psi_{A\cup B}=\frac{|\partial(A\cup B)|}{|A\cup B|}=\frac{|\partial
A|+|\partial B|-2}{|A|+|B|}\geq \min\left\{\frac{|\partial
A|}{|A|},\frac{|\partial B|}{|B|}\right\}-\frac{2}{|A|+|B|},
\end{aligned}
\end{equation}
and so the lemma follows.
\end{proof}

\begin{proof}[Proof of Claim \ref{clm:nabphi}]
We separate the proof into six different cases according to the
following table:

\begin{center}
\begin{tabular}{|c|c|c|}\hline
e=(x,y) & $\begin{array}{c}
\omega(e)=0\\
(\omega=\hat{\omega}^{e})\end{array}$ & $\begin{array}{c}
\omega(e)=1\\
(\omega=\check{\omega}^{e})\end{array}$ \\\hline $x,y\notin C_d(n)$
& 1 & 2 \\\hline $x,y\in C_d(n)$ & 3 & 4 \\\hline $\begin{array}{c}
x\in C_{d}(n)\,,\, y\notin C_{d}(n)\\
\mbox{or}\\
y\in C_{d}(n)\,,\, x\notin C_{d}(n)\end{array}$ & 5 & 6 \\\hline
\end{tabular}
\end{center}

\begin{figure}
  \centering
  \subfloat[Case 1]{\label{fig:case1}\includegraphics[width=0.36\textwidth]{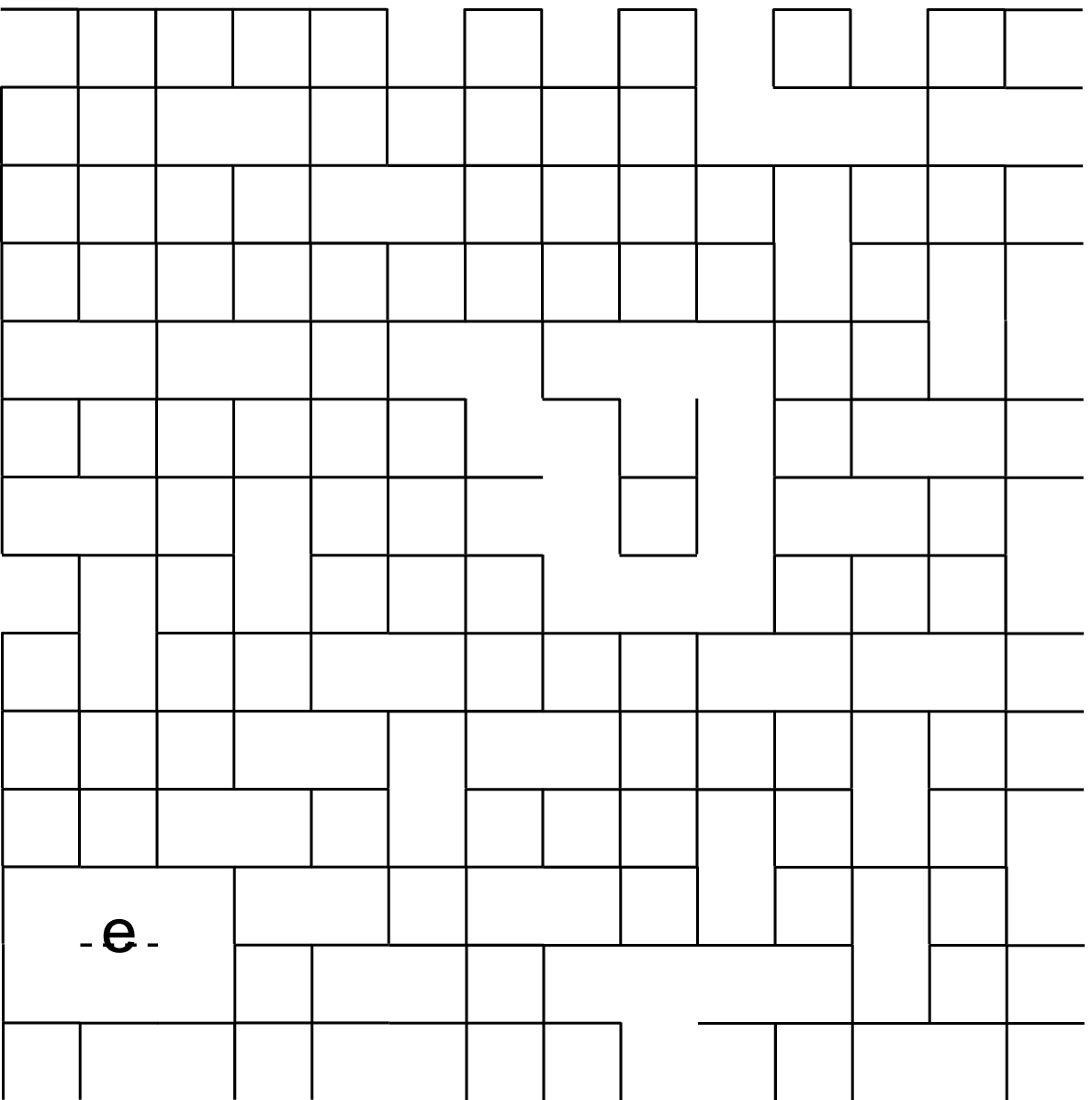}}\qquad
  \subfloat[Case 2]{\label{fig:case2}\includegraphics[width=0.36\textwidth]{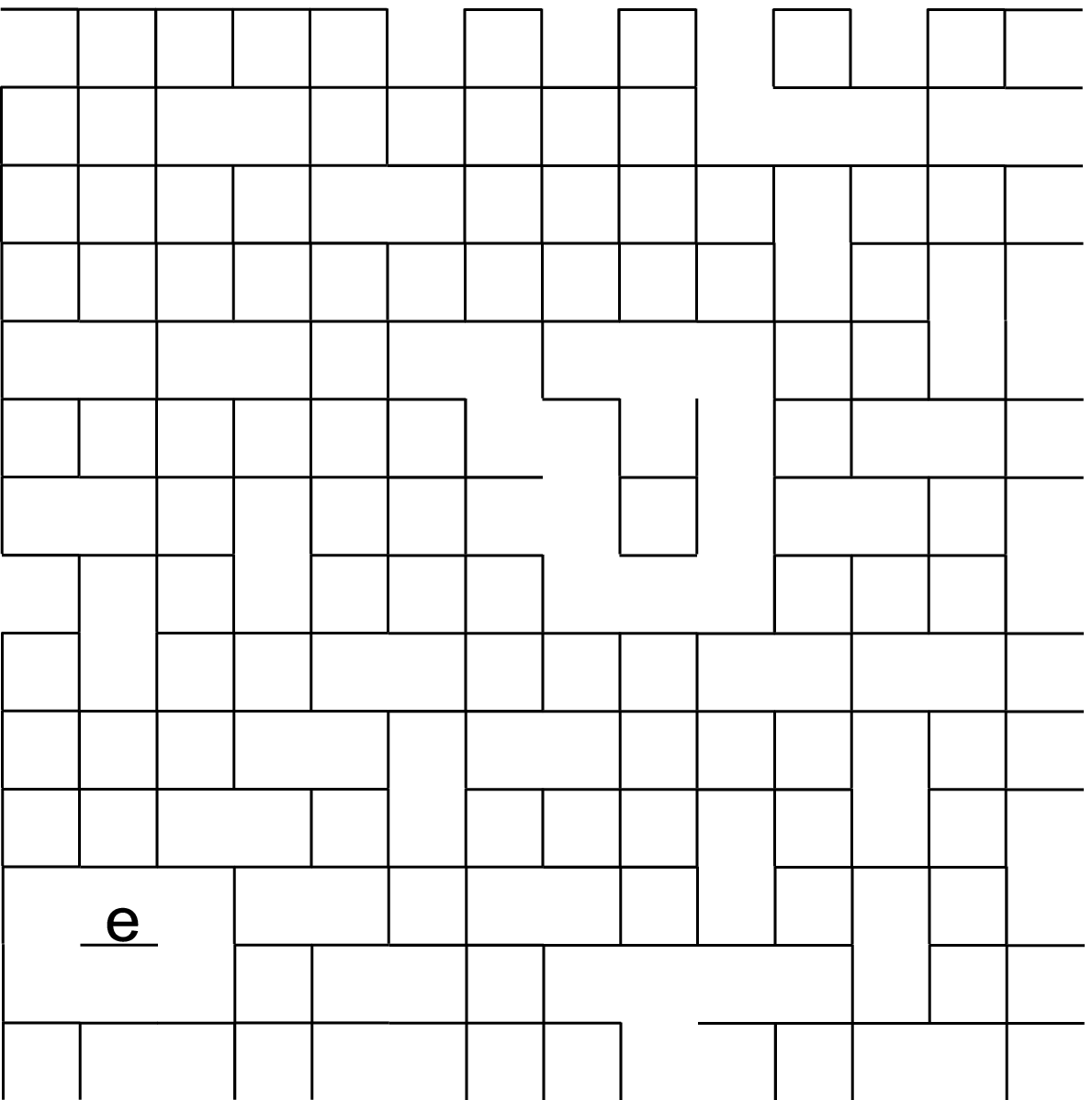}}\\
  \subfloat[Case 3]{\label{fig:case3}\includegraphics[width=0.36\textwidth]{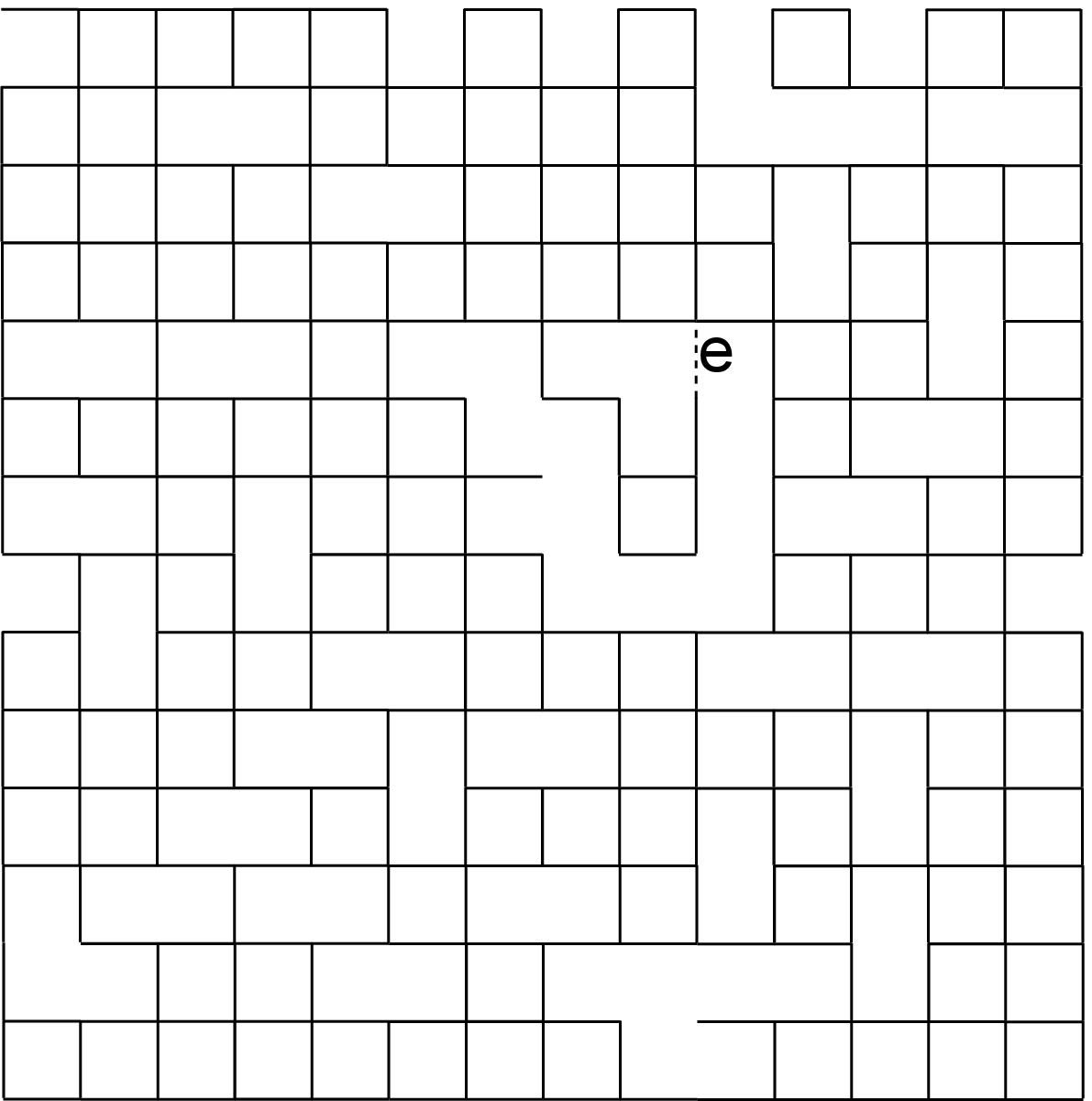}}\qquad
  \subfloat[Case 4a]{\label{fig:case4a}\includegraphics[width=0.36\textwidth]{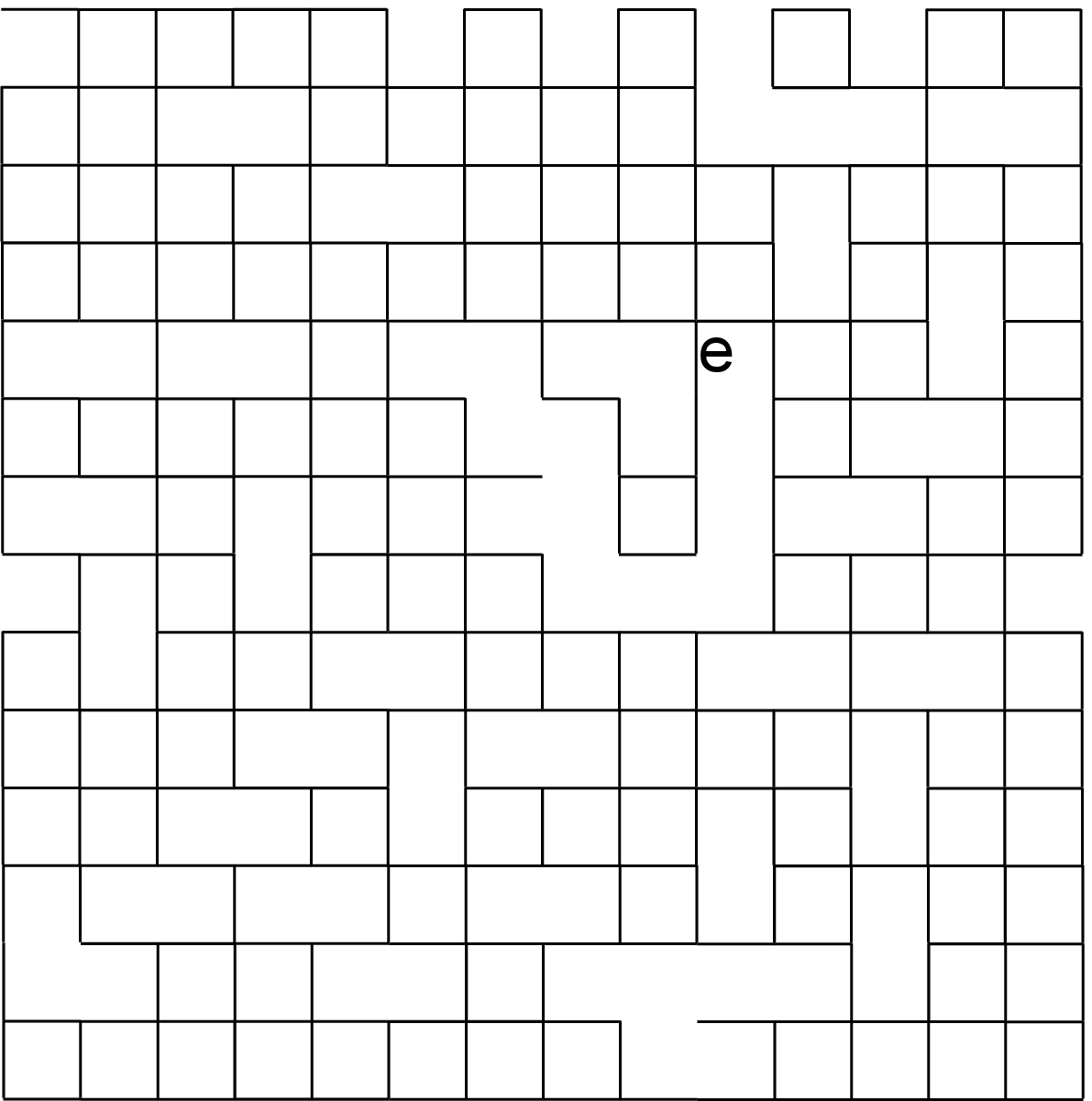}}\\
  \subfloat[Case 4b]{\label{fig:case4b}\includegraphics[width=0.36\textwidth]{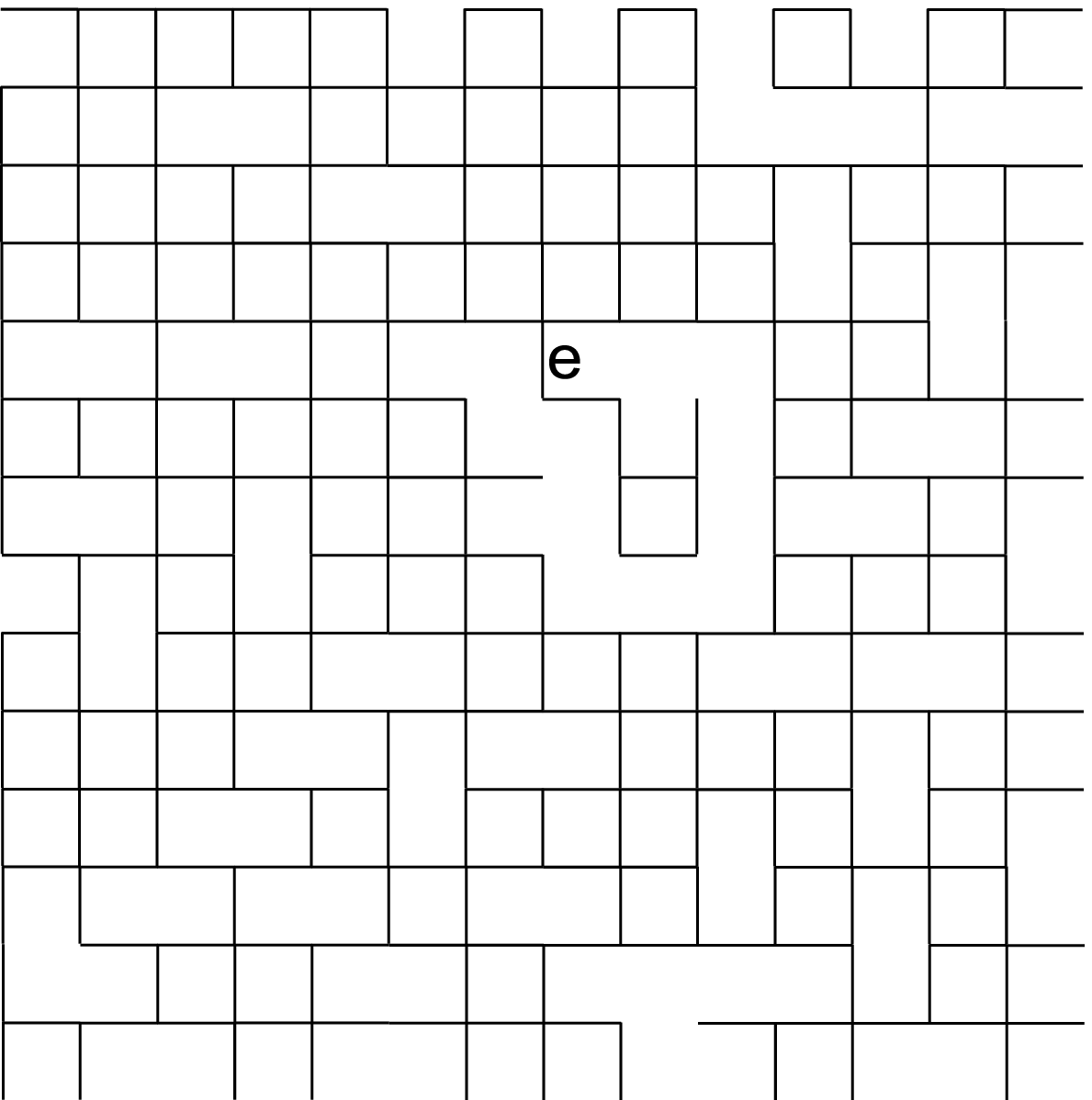}}\qquad
  \subfloat[Case 5]{\label{fig:case5}\includegraphics[width=0.36\textwidth]{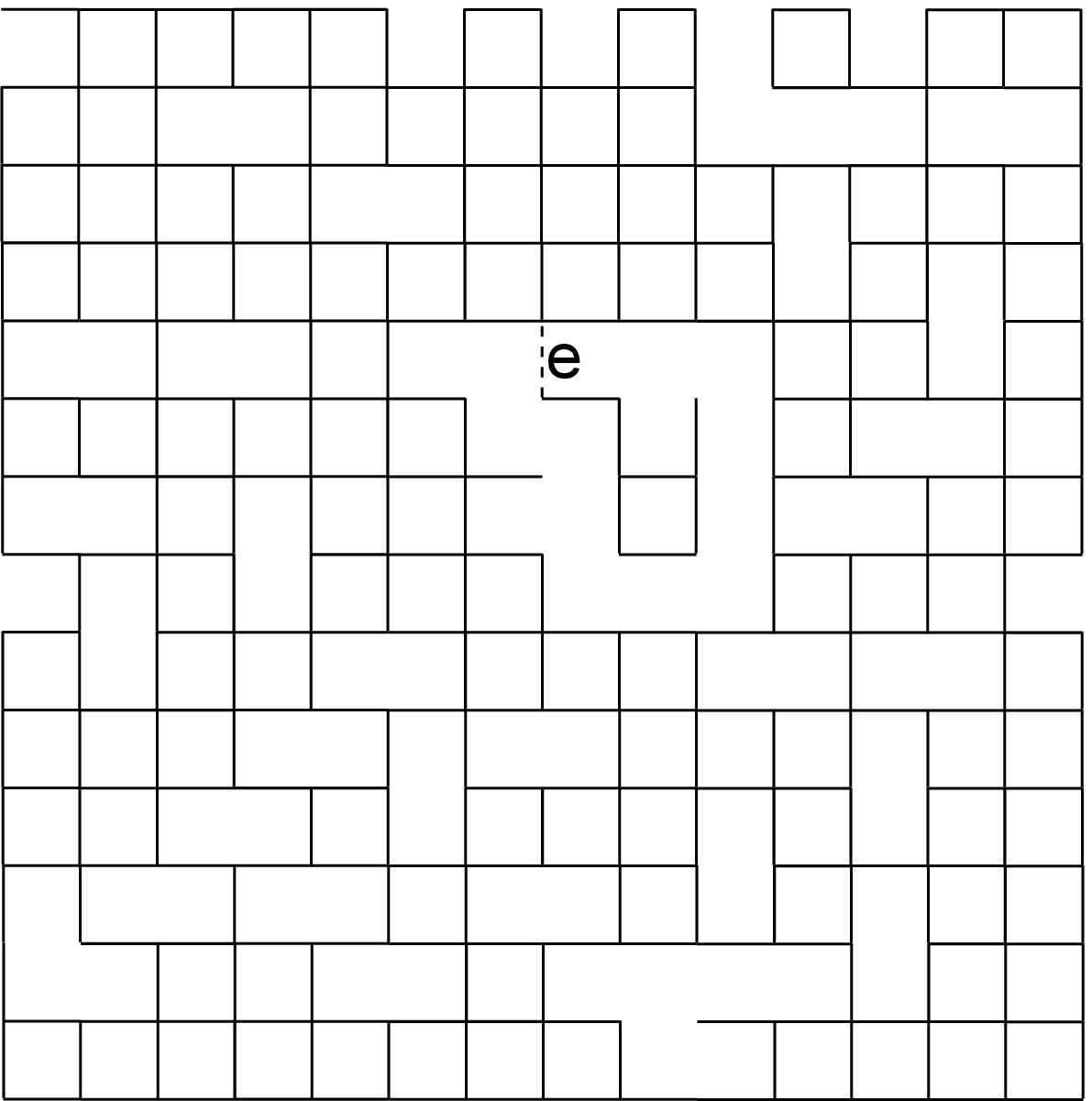}}
  \caption{Illustrations of the different cases}
  \label{fig:Cases}
\end{figure}

%\begin{figure}%
%\centering
% \subfigure[a]{
% \includegraphics[width=0.3\textwidth]{torusperccase1}
% \caption{Case 1\label{fig:case1}}}%
% \subfigure[ssss]{
% \includegraphics[width=0.3\textwidth]{torusperccase2}
% \caption{Case 2\label{fig:case2}}}
%\end{figure}

\begin{itemize}
\item \textbf{Cases 1 and 2: } In those cases the set $C_d(n)$ and the edges available from it is the same for both configurations $\om$ and $\om^e$. It therefore follows
that $\nabla_e\phi(\om)=0$. See Figure \ref{fig:case1}, and \ref{fig:case2}.
\item \textbf{Case 3:} In this case the set $C_d(n)$ is the same for both configurations $\om$ and $\om^e$, however the set of edges available from $C_d(n)$ is
increased by one when moving to the configuration $\om^e$, see figure \ref{fig:case3}. Fix a set
$A\subset C_d(n)(\om)$ of size bigger than $c_4n^d$ which realize
the Cheeger constant. It follows that
\[
\psi_A(\om)=\phi(\om)\leq \phi(\om^e)\leq \psi_A(\om^e),
\]
and therefore by Lemma \ref{lem:isobigset} we have
\[
|\phi(\om^e)-\phi(\om)|\leq \psi_A(\om^e)-\psi_A(\om)\leq
\frac{1}{c_4 n^d},
\]
as required.
\item \textbf{Case 4:} We separate this case into two subcases
according to the fact weather $C_d(n)(\om)\backslash C_d(n)(\om^e)$
is an empty set or not. If $C_d(n)(\om)\backslash
C_d(n)(\om^e)=\emptyset$ then we are in the same situation as in
{\bf Case 3}, see Figure \ref{fig:case4a}, and so the same argument gives the desired result. So,
let us assume that $C_d(n)(\om)\backslash
C_d(n)(\om^e)\neq\emptyset$, see Figure \ref{fig:case4b}. Since $\om\in H_n$ we know that
\begin{equation}\label{changeinsizeofC}
\begin{aligned}
\left|C_d(n)(\om)\backslash C_d(n)(\om^e)\right|\leq \sqrt{n}.
\end{aligned}
\end{equation}
Since $\om\in H_n^4$ there exists a set $A\subset C_d(n)(\om)$ of size bigger than $c_4 n^d$
realizing the Cheeger constant in the configuration $\om$. We denote
$A_1=A\cap C_d(n)(\om^e)$ and $A_2=A\cap (C_d(n)(\om)\backslash
C_d(n)(\om^e))$. Applying Lemma \ref{lem:twocom} to $A_1$ and $A_2$
we see that
\begin{equation}\label{eq:psia}
\begin{aligned}
\psi_A(\om)=\psi_{A_1\cup A_2}(\om)\geq
\min\{\psi_{A_1}(\om),\psi_{A_2}(\om)\}-\frac{2}{|A|}.
\end{aligned}
\end{equation}
From \eqref{changeinsizeofC} it follows that $|A_2|\leq \sqrt{n}$
and therefore $\psi_{A_2}(\om)\geq \frac{1}{\sqrt{n}}$ which gives
us that $\min\{\psi_{A_1}(\om),\psi_{A_2}(\om)\}=\psi_{A_1}(\om)$.
Indeed, if the last equality doesn't hold then
\[
\frac{c_2}{n}\geq \psi_A(\om)\geq \psi_{A_2}(\om)-\frac{2}{|A|}\geq
\frac{1}{\sqrt{n}}-\frac{2}{c_4 n^d},
\]
which for large enough $n$ yields a contradiction. Consequently from \eqref{eq:psia} we
get that
\[
\psi_{A_1}(\om)-\frac{2}{c_4 n^d}\leq  \phi(\om)\leq
\psi_{A_1}(\om),
\]
and so
\[
\phi(\om^e)-\frac{2}{c_4n^d}\leq
\psi_{A_1}(\om^e)-\frac{2}{c_4n^d}\leq
\psi_{A_1}(\om)-\frac{2}{c_4n^d}\leq \phi(\om),
\]
i.e, $\phi(\om^e)-\phi(\om)\leq \frac{2}{c_4n^d}$.

For the other direction, since $\om\in H_n^5$ there exists a set $B\subset C_d(n)(\om^e)$ of size bigger than
$c_5 n^d$ realizing the Cheeger constant in $\om^e$, then
\[
\phi(\om)\leq \psi_B(\om)\leq
\psi_B(\om^e)+\frac{1}{|B|}=\phi(\om^e)+\frac{1}{|B|}\leq
\phi(\om^e)+\frac{1}{c_5n^d}.
\]
Consequently,
\[
|\phi(\om)-\phi(\om^e)|\leq
\max\left\{\frac{2}{c_4n^d},\frac{1}{c_5n^d}\right\},
\]
as required.
\item \textbf{Case 5: } This case is similar to \textbf{Case 4}, see Figure \ref{fig:case5}. The proof of this case follows the proof of \textbf{case 4} above.
\item \textbf{Case 6: } This case is impossible by the definition of
the set $C_d(n)(\om)$.
\end{itemize}
\end{proof}

Next we turn to estimate the probability of the event $H_n$.

\begin{clm}\label{clm:lotsofconts}
There exist constants $c_1,c_2,c_3,c_4,c_5>0$ and a constant $c>0$ such that for large enough $n\in\BN$ we have
\begin{equation}\begin{aligned}
\prob(H_n^c)\leq  e^{-c \log^{\frac{3}{2}}(n)}.
\end{aligned}\end{equation}
\end{clm}

\begin{proof}
Since $\prob(H_n^c)\leq \sum_{i=1}^{5}\prob((H_n^i)^c)$, it's enough to bound each of the last probabilities.
The proof of the exponential decay of $\prob((H_n^1)^c)$ for appropriate constant is presented in the Appendix.

By \cite{mathieu2004isoperimetry} Theorem 3.1 and section 3.4, there exists a $c>0$ such that for $n$ large enough,
$\prob((H_n^2)^c)\le e^{-c\log^{3/2}n}$ for some constants $c_2,c_3>0$.

Turning to bound $\prob((H_n^3)^c)$, we notice that the set $C_d(n)(\om)\triangle C_d(n)(\om^e)$
is independent of the status of the edge $e$ and therefore
\begin{equation}\begin{aligned}
\prob((H_n^3)^c) &=\frac{1}{1-p}\prob\left(\left\{\om\in\Om
~:~\exists e\in E_d(n)\quad \left|C_d(n)(\om)\triangle
C_d(n)(\om^e)\right|\geq\sqrt{n} ~,~e \text{ is closed}\right\}\right)\\
&\leq \frac{1}{1-p}\prob\left(\left\{\om\in\Om ~:~\exists e\in
E_d(n)\quad \left|C_d(n)(\om)\triangle
C_d(n)(\om^e)\right|\geq\sqrt{n} ~,~e \text{ is closed}\right\}\cap
H_n^1\right)\\&+\frac{1}{1-p}\prob((H_n^1)^c).
\end{aligned}\end{equation}
We already gave appropriate bound for the last term and therefore we
are left to bound the probability of $\left\{\om\in\Om ~:~\exists
e\in E_d(n)\quad \left|C_d(n)(\om)\triangle
C_d(n)(\om^e)\right|\geq\sqrt{n} ~,~e \text{ is closed}\right\}\cap
H_n^1$. Notice that the occurrence of this event implies the
existence of an open cluster of size bigger than $\sqrt{n}$ which is
not connected to $C_d(n)$, and therefore its probability is bounded
by
\[
\prob(F\cap H_n^1):=\prob(\{\exists B,|B|\ge\sqrt{n}, B\text{ is an open cluster that is not connected to }C_d(n)\}\cap H_n^1)
.\]
By \cite{mathieu2004isoperimetry} Appendix B and \cite{grimmett1999percolation} Theorem 8.61, for large enough $n$ we have that,
\begin{equation}\begin{aligned}
\prob(F\cap H_n^1)\le\prob\left(\begin{array}{cc}\exists B,|B|\ge\sqrt{n}, B\text{ is an open cluster that}\\
\text{is not connected to the infinite cluster}\end{array}\right).
\end{aligned}\end{equation}
However the probability of the last event decays exponentially with $n$ by \cite{grimmett1999percolation} Theorem 8.18.

In order to deal with the event $(H_n^4)^c$ we define one last event
\[
G_n=\left\{I_{\ep(n)}(C_d(n))\geq c_6 n^{\frac{d}{\ep(n)}-1}\right\},
\]
where $\ep(n)=d+2d\frac{\log\log n}{\log n}$ and
\begin{equation}
I_{\ep}(C_d(n))=\min_{\begin{array}{ll}&_{\emptyset\neq A\subset C_d(n)}\\&_{|A|\le|C_d(n)|/2}\end{array}}
\frac{|\partial_{C_d(n)}A|}{|A|^{(\ep-1)/(\ep)}}.
\end{equation}

By \cite{mathieu2004isoperimetry} there exists a constant $c>0$ such that for large enough $n\in\BN$
$\prob(G_n^c)<e^{-c\log^\frac{3}{2}n}$. As before we write
\[
\prob((H_n^4)^c)\leq \prob((H_n^4)^c\cap H_n^1\cap H_n^2\cap G_{n})+\prob((H_n^1)^c\cup (H_n^2)^c\cup G_{n}^c),
\]
and by the probability bound mentioned so far it's enough to bound the probability of the first event $(H_n^4)^c\cap H_n^1\cap H_n^2\cap G_{n}$.
What we will actually show is that for appropriate choice of $0<c_4<\frac{1}{2}$ we have
$(H_n^4)^c\cap H_n^1\cap H_n^2\cap G_{n}=\emptyset$. Indeed, since we assumed the event $G_n$ occurs
we have that for large enough $n\in\BN$ and every set $A\subset C_d(n)(\om)$ of size smaller than $c_4 n^d$
\[
|\partial_{C_d(n)}A|\ge c_6
n^{\frac{d}{\ep(n)}-1}|A|^\frac{\ep(n)-1}{\ep(n)}.
\]
It follows that
\begin{equation}
\begin{aligned}
\psi_A \ge c_6 n^{\frac{d}{\ep(n)}-1}\frac{1}{|A|^{1/\ep(n)}}\ge
c_6 n^{\frac{d}{\ep(n)}-1}\frac{1}{c_4^{1/\ep(n)}n^{d/\ep(n)}}=\frac{c_6}{c_4^{1/\ep(n)}n}.
\end{aligned}
\end{equation}
Choosing $c_4>0$ such that for large enough $n\in\BN$ we have $\frac{c_6}{c_4^{1/\ep(n)}}>c_3$,
we get a contradiction to the event $H_n^2$, which proves that the event is indeed empty.

Finally we turn to deal with the event $(H_n^5)^c$. As before it's
enough to bound the probability of the event $(H_n^5)^c\cap
H_n^1\cap H_n^2\cap H_n^3\cap H_n^4 \cap G_{n}$. We divide the last
event into two disjoint events according to the status of the edge
$e$, namely
\begin{equation}\begin{aligned}
V_n^0:&=(H_n^5)^c\cap H_n^1\cap H_n^2\cap H_n^3\cap H_n^4 \cap G_{n}\cap \{\om(e)=0\}\\
V_n^1:&=(H_n^5)^c\cap H_n^1\cap H_n^2\cap H_n^3\cap H_n^4 \cap
G_{n}\cap \{\om(e)=1\},
\end{aligned}\end{equation}
and will show that for right choice of $c_5>0$ both $V_n^0$ and
$V_n^1$ are empty events.

Let us start with $V_n^0$. Going back to the proof of Claim
\ref{clm:nabphi} one can see that under the event $H_n^1\cap
H_n^2\cap H_n^3\cap H_n^4$ there exists a constant $c>0$ such that
\begin{equation}\begin{aligned}\label{eq:Cheeger_om^e}
\phi(\om^e)\leq \phi(\om)+\frac{c}{n^d}\leq
\frac{c_3}{n}+\frac{c}{n^d},
\end{aligned}\end{equation}
and therefore $\phi(\om^e)\leq \frac{\tilde{c}_3}{n}$ for any
$\tilde{c}_3>c_3$ and $n\in\BN$ large enough. If $\emptyset\neq
A\subset C_d(n)(\om^e)$ is a set of size smaller than
$\frac{n}{\tilde{c}_3}$ then
\begin{equation}
\psi_A(\om^e)\geq \frac{1}{|A|}>\frac{\tilde{c}_3}{n},
\end{equation}
and therefore $A$ cannot realize the Cheeger constant. On the other
hand, if $A\subset C_d(n)(\om^e)$ satisfy $\frac{n}{\tilde{c}_3}\leq
|A|\leq c_5 n^d$ then
\[
|\partial_{C_d(n)(\om^e)}A|\geq |\partial_{C_d(n)(\om^e)}(A\cap
C_d(n)(\om))|-1 \geq |\partial_{C_d(n)(\om)}(A\cap C_d(n)(\om)|-2,
\]
and therefore (Since we assumed the event $G$ occurs)
\begin{equation}\begin{aligned}
\psi_A(\om^e)&\geq \frac{|\partial_{C_d(n)(\om)}(A\cap
C_d(n)(\om))|-2}{|A|}\\
&\geq c_6 n^{d/\ep(n)-1}\frac{|A\cap C_d(n)(\om)|}{|A|}-\frac{2}{|A|}
\geq \frac{c_6}{2c_5^{1/\ep(n)}n}-\frac{2\tilde{c}_3}{n}.
\end{aligned}\end{equation}
Taking $c_5>0$ small enough such that $\frac{c_6}{2c_5^{1/\ep(n)}}-2\tilde{c}_3>\tilde{c}_3$ we get a contradiction to \eqref{eq:Cheeger_om^e}.
It follows that no set $A\subset C_d(n)(\om^e)$ of size smaller than $c_5 n^d$ can realize the Cheeger constant which contradicts $(H_n^5)^c$, i.e,
$V_n^0=\emptyset$.

Finally, for $V_n^1$. The case $A\subset C_d(n)(\om^e)$ such that $|A|< \frac{n}{\tilde{c}_3}$
is the same as for the event $V_n^0$. If $A\subset C_d(n)(\om^e)$ satisfy $\frac{n}{\tilde{c}_3}\leq |A|\leq c_5 n^d$
then
\[
|\partial_{C_d(n)(\om^e)}A|\geq |\partial_{C_d(n)(\om)}A|-1.
\]
and therefore as in the case of $V_n^0$
\begin{equation}\begin{aligned}
\psi_A(\om^e)&\geq \frac{|\partial_{C_d(n)(\om)}A|-1}{|A|}\\
&\geq c_6 n^{d/\ep(n)-1}\frac{|\partial_{C_d(n)(\om)}A|}{|A|}-\frac{1}{|A|} \geq \frac{c_6}{2c_5^{1/\ep(n)}n}-\frac{\tilde{c}_3}{n}.
\end{aligned}\end{equation}
Choosing $c_5$ small enough, we again get a contradiction to \eqref{eq:Cheeger_om^e}.
and as before this yields that $V_n^1=\emptyset$.
\end{proof}

\begin{proof}[Proof of theorem \ref{thm:main}]
By \cite{talagrand1994russo} (Theorem 1.5) the following inequality holds for some $K=K(p)$,
\begin{equation}
\begin{aligned}
\var(\phi)\leq K\cdot \sum_{e\in E(C_d(n))}\frac{\|\nabla_e\phi\|_2^2}{1+\log\left({\|\nabla_e\phi\|_2}/{\|\nabla_e\phi\|_1}\right)}.
\end{aligned}
\end{equation}

Where $\|\nabla_e\phi\|_2^2=\ev[(\nabla_e\phi)^2]$ and $\|\nabla_e\phi\|_1=\ev[|\nabla_e\phi|]$. Observe that $\|\nabla_e\phi\|_1=\|\nabla_e\phi\ind_{\{\nabla_e\phi\neq 0\}}\|_1\le
\|\nabla_e\phi\|_2\|\ind_{\{\nabla_e\phi\neq 0\}}\|_2$, and therefore
\[
\frac{\|\nabla_e\phi\|_2}{\|\nabla_e\phi\|_1}\ge\frac{1}{\sqrt{\pr(\nabla_e\phi\neq0)}}\ge1
.\]
Consequently,
if we fix some edge $e_0\in\BE_d(n)$,
\begin{equation}\label{eq:main_thm_ineq}
\begin{aligned}
\var(\phi)\leq K\sum_{e\in E(C_d(n))}\|\nabla_e\phi\|_2^2=K|\BE_d(n)|\cdot \|\nabla_{e_0}\phi\|_2^2=Kdn^d\cdot\|\nabla_{e_0}\phi\|_2^2.
\end{aligned}
\end{equation}
where the first equality follows from the symmetry of $\BT_d(n)$.
\begin{equation}
\begin{aligned}
\|\nabla_{e_0}\phi\|_2^2=\ev[|\nabla_{e_0}\phi|^2\ind_{H_n}]+\ev[|\nabla_{e_0}\phi|^2\ind_{H_n^c}].
\end{aligned}
\end{equation}
Notice that since $|\nabla_{e_0}\phi|\le4d$ we have $\ev[|\nabla_{e_0}\phi|^2\ind_{H_n^c}]\le16d^2\pr(H_n^c)$. Thus applying Lemma \ref{clm:lotsofconts},
\begin{equation}\label{eq:nabsmall}
\begin{aligned}
\ev[|\nabla_{e_0}\phi|^2\ind_{H_n^c}]\le16d^2e^{-c
\log^{\frac{3}{2}}(n)} ,
\end{aligned}
\end{equation}
and by Lemma \ref{clm:nabphi}
\begin{equation}\label{eq:evnab}
\begin{aligned}
\ev[|\nabla_{e_0}\phi|^2\ind_{H_n}]\le\frac{C^2}{n^{2d}} .
\end{aligned}
\end{equation}
Thus combing equations \eqref{eq:nabsmall} and \eqref{eq:evnab} with
equation \eqref{eq:main_thm_ineq} the result follows.
\end{proof}

\section{Appendix}\label{sec:appendix}
In this Appendix for completeness and future reference  we sketch a proof of the
exponential decay of $\prob((H_n^1)^c)$. The proof follows directly from two
papers \cite{deuschel1996surface} by Deuschel and Pistorza and \cite{antal1996chemical} by Antal Pisztora,
which together gives a proof by a renormalization argument. We borrow the terminology of \cite{antal1996chemical} without giving here the definitions.
\begin{lem}
Let $p>p_c(\Z^d)$. There exists a $c_1,c>0$ such that for $n$ large enough \[\prob_p(|C_d(n)(\om)|<c_1 n^d)<e^{-cn}.\]
\begin{proof}
By \cite{deuschel1996surface} Theorem 1.2, there exists a
$p_c(\Z^d)<p^*<1$ such that for every $p>p^*$,
$\prob_p(|C_d(n)(\om)|< \tilde{c}_1n^d)<e^{-cn}$. Since $\{|C_d(n)(\om)|<\tilde{c}_1
n^d\}^c$ is an increasing event, by Proposition 2.1 of
\cite{antal1996chemical} for $N\in\BN$ large enough, i.e such that
$\bar{p}(N)>p^*$,
\begin{equation}
\begin{aligned}
\pr_N(|C_d(n)(\om)|<\tilde{c}_1
n^d)\le\pr^*_{\bar{p}(N)}(|C_d(n)(\om)|<\tilde{c}_1 n^d)<e^{-cn},
\end{aligned}
\end{equation}
where $\BP_N$ is the probability measure of the renormalized
dependent percolation process and $\BP_{\bar{p}(N)}^*$ is the
probability measure of standard bond percolation with parameter
$\bar{p}(N)$.  From the definition of the event $R_i^{(N)}$, the
crossing clusters of all the boxes $B_i'$ that admit $R_i^{(N)}$ are
connected, thus
\[
\prob_p\left(|C_d(n)(\om)|<\frac{\tilde{c}_1}{N^d} n^d \right)<e^{-cn}.\]
\end{proof}
\end{lem}

\bibliography{ns}
\bibliographystyle{alpha}
\end{document}